\documentclass[final,a4paper,11pt]{article}
\usepackage{graphicx}
\usepackage{mathptmx}
\usepackage{amssymb,amsmath, amsfonts,amsthm}
\usepackage{a4wide}
\usepackage{color}

\usepackage{url}
\usepackage{tikz}
\usepackage{pgfplots}
\pgfplotsset{compat=1.14}
\usepackage{comment}
\usepackage{cases}
\usepackage{float}
\usepackage{array}
\usepackage{tabularx}
\usepackage{ragged2e}
\usepackage{comment}
\newcommand{\red}[1]{\textcolor{red}{#1}}

\newcommand{\email}[1]{{\tt #1}}
\newcommand{\R}{\mathbb{R}}

\newcommand{\norm}[1]{\|#1\|}

\newcommand{\dist}[1]{{\rm dist}(#1)}

\newcommand{\mv}{\,\big |\,}
\newcommand{\bmv}{\,\Big |\,}

\newcommand{\B}{{\cal B}}

\newcommand{\K}{{\cal K}}
\newcommand{\M}{{\cal M}}

\newcommand{\Sp}{{\mathcal S}}
\newcommand{\F}{{\cal F}}
\newcommand{\Z}{{\cal Z}}

\newcommand{\setto}[1]{\mathop{\rightarrow}\limits^#1}
\newcommand{\longsetto}[1]{\mathop{\longrightarrow}\limits^{#1}}
\newcommand{\skalp}[1]{\langle #1\rangle}

\newcommand{\xb}{\bar x}
\newcommand{\yb}{\bar y}
\newcommand{\zb}{\bar z}

\newcommand{\AT}[2]{{\textstyle{#1\atop#2}}}

\newcommand{\OO}{{\cal O}}

\newcommand{\cl}{{\rm cl\,}}

\newcommand{\co}{{\rm conv\,}}
\newcommand{\gph}{\mathrm{gph}\,}

\newcommand{\tto}{\rightrightarrows}
\newcommand{\Limsup}{\mathop{{\rm Lim}\,{\rm sup}}}
\newcommand{\myvec}[1]{\begin{pmatrix}#1\end{pmatrix}}
\newcommand{\SCD}{SCD\ }

\newcommand{\orth}{{\rm orth\,}}
\newcommand{\ssstar}{semismooth$^{*}$ }

\newcommand{\rge}{{\rm rge\;}}

\newlength{\myAlgBox}
\newtheorem{theorem}{Theorem}[section]
\newtheorem{proposition}[theorem]{Proposition}

\newtheorem{lemma}[theorem]{Lemma}
\newtheorem{corollary}[theorem]{Corollary}
\newtheorem{definition}[theorem]{Definition}
\newtheorem{example}[theorem]{Example}

\title{On the isolated calmness property of implicitly defined multifunctions}
\if{\title{On the numerical solution of a class of contact problems with Coulomb friction via the SCD \ssstar  Newton method }}\fi
\author{Helmut Gfrerer\thanks{Institute of Computational Mathematics, Johannes Kepler University
Linz, A-4040 Linz, Austria; \email{helmut.gfrerer@jku.at}}
 \and   Ji\v{r}\'{i} V. Outrata\thanks{Institute of Information Theory and Automation, Czech Academy of
 Sciences, 18208 Prague, Czech Republic, and Centre for
              Informatics and Applied Optimization, Federation University of Australia, POB 663,
              Ballarat,  Vic 3350, Australia,  \email{outrata@utia.cas.cz}}}
\date{}

\begin{document}
\maketitle
\begin{abstract}
    The paper deals with an extension of the available theory of SCD (subspace containing derivatives) mappings to mappings between spaces of different dimensions. This  extension enables us to derive workable sufficient conditions for the isolated calmness of implicitly defined multifunctions {\em around}  given reference points. This stability property differs substantially from isolated calmness {\em at} a point and, possibly in conjunction with the Aubin property, offers a new useful stability concept. The application area includes a broad class of parameterized generalized equations, where the respective conditions ensure a rather strong type of Lipschitztan behavior of their solution maps.
\end{abstract}



\section{Introduction}
Analysis of Lipschitzian stability of set-valued mappings is one of the most important parts of modern variational analysis. Above all, the notions of the Aubin and the calmness property play a central role both in parameter-dependent equilibria (especially in presence of unknown parameters) and in qualification conditions of generalized differential calculus. But also the so-called isolated calmness and (the existence of) single-valued Lipschitzian localization have a great importance, e.g., in connection with Newton-type methods for nonsmooth problems.

There are various pointwise characterizations of the above mentioned stability notions in terms of generalized derivatives as, e.g., the Mordukhovich or the Levy-Rockafellar criteria.

Recently, in connection with the so-called SCD (subspace containing derivatives) mappings and the associated SCD \ssstar Newton method in \cite{GfrOut22a}, the authors derived for such mappings a characterization of the strong metric subregularity on a neighborhood which amounts (\cite[Theorem 3I.2]{DoRo14}) to the isolated calmness on a neighborhood of their inverses.

Both these properties differ from their counterparts {\em at a point} rather substantially and one obtains thus a useful amendment to the available arsenal of regularity and stability properties. In \cite{GfrOut22a}, one finds both a characterization of the strong metric subregularity on a neighborhood for general mappings and special mappings having SCD and \ssstar properties. These characterizations of strong metric subregularity, however, were presented in \cite{GfrOut22a} only for mappings between spaces of the same dimension. So, in order to derive workable criteria for the isolated calmness around a reference point for, say, a class of implicitly given multifunctions, the basic framework has to be extended. This  extension, along with the corresponding stability results, is the aim of the present paper.

The plan of the paper is as follows. In the next section one finds a necessary background from variational analysis which is used throughout the whole paper. Section 3 contains the basic elements of the theory of SCD mappings between spaces of different dimensions. In this development one uses the corresponding part of   \cite{GfrOut22a} as a template. In Section 4 several calculus rules are derived which are needed in the proofs of the stability results presented in Section 5. The main statement (Theorem \ref{Th5.1}) provides us with two types of conditions ensuring that the implicit multifunction, defined via the inclusion
\[
0 \in H(x,y)
\]
 possesses the {\em isolated calmness property on a neighborhood} of the given {\em reference point}. One of these conditions is based on the so-called {\em outer limiting graphical derivative} and works for general mappings $H$, whereas the other one is tailored to \ssstar SCD mappings and is available in a primal and a dual form.  To illustrate the nature of these conditions, we use a class of parameterized {\em generalized equations} (GEs). In case of variational inequalities with polyhedral constraint sets, we work out these conditions to an efficient form expressed in terms of faces of the critical cone to the constraint set. It appears that the specialized condition is much simpler to verify than the general one.
 In addition, we present in Section 5 another condition expressed in terms of the limiting (Mordukhovich) coderivative of $H$ which ensures that the respective implicit mapping has both the Aubin and the isolated calmness property around the reference point.

The following notation is employed. Given a linear subspace $L\subseteq \R^n$, $ L^\perp$ denotes its orthogonal
complement and, for a closed cone $K$ with vertex at the origin, $K^\circ$ signifies its (negative) polar.
Further, given a multifunction $F$, $\gph F:=\{(x,y)\mv y\in F(x)\}$ stands for its
graph. For an element $u\in\R^n$, $\norm{u}$ denotes its Euclidean norm and  $\B_\delta(u)$ denotes the closed ball around $u$
with radius $\delta$. The closed unit ball in $\R^n$ is denoted by $\B_{\R^n}$. In a product space we use the norm $\norm{(u,v)}:=\sqrt{\norm{u}^2+\norm{v}^2}$. Given
an $m\times n$ matrix $A$, we employ the operator norm $\norm{A}$ with respect to the Euclidean norm and we denote the range of A by $\rge A$. Given a set $\Omega\subset\R^s$, we define the distance of a point $x$ to $\Omega$ by $d_\Omega(x):=\dist{x,\Omega}:=\inf\{\norm{y-x}\mv y\in\Omega\}$ and the indicator function is denoted by $\delta_\Omega$. Finally, $x\setto{\Omega}\xb$ denotes comvergence within the set $\Omega$.  When a mapping $F:\R^n\to\R^m$ is differentiable at $x$, we denote by $\nabla F(x)$ its Jacobian.

\section{Background from variational analysis}
Throughout the whole paper, we will frequently use  the following basic notions of modern
variational analysis. All the sets under consideration are supposed to be {\em locally closed} around the points in question without further mentioning.
 \begin{definition}\label{DefVarGeom}
 Let $A$  be a  set in $\mathbb{R}^{s}$ and let $\bar{x} \in A$. Then
\begin{enumerate}
 \item [(i)]The  {\em tangent (contingent, Bouligand) cone} to $A$ at $\bar{x}$ is given by
 \[T_{A}(\bar{x}):=\Limsup\limits_{t\downarrow 0} \frac{A-\bar{x}}{t},\]
 the {\em paratingent cone} to $A$ at $\xb$ is given by
 \[T^P_A(\xb):=\Limsup\limits_{\AT{x\setto{{A}}\xb}{t\downarrow 0}} \frac{A- x}{t}\]
 and the {\em outer limiting tangent cone} to $A$ at $\xb$ is defined as
     \begin{equation}\label{EqLimTanCone}
      T^\sharp_A(\xb):=\Limsup_{x\setto{{A}}\xb} T_A(x)= \Limsup_{x\setto{{A}}\xb}\Big(\Limsup_{t\downarrow 0} \frac {A-x}t\Big).
    \end{equation}

 \item[(ii)] The set
 \[\widehat{N}_{A}(\bar{x}):=\big(T_{A}(\bar{x})\big)^{\circ}\]
 is the {\em regular (Fr\'{e}chet) normal cone} to $A$ at $\bar{x}$, and

 \[N_{A}(\bar{x}):=\Limsup\limits_{\stackrel{A}{x \rightarrow \bar{x}}} \widehat{N}_{A}(x)\]
 is the {\em limiting (Mordukhovich) normal cone} to $A$ at $\bar{x}$.
 \end{enumerate}

\end{definition}
In this definition ''$\Limsup$'' stands for the Painlev\' e-Kuratowski {\em outer (upper) set limit}, see, e.g.,\cite{AubFra90}.
The outer limiting tangent cone $T^\sharp_A(\xb)$ was very recently defined in \cite{GfrOut22a} and it is always contained in the paratingent cone $T^P_A(\xb)$. All the other objects from variational geometry are well-known and can be found in standard textbooks, see, e.g., \cite{RoWe98}.

If $A$ is convex, then $\widehat{N}_{A}(\bar{x})= N_{A}(\bar{x})$ amounts to the classical normal cone in
the sense of convex analysis and we will  write $N_{A}(\bar{x})$.

The above listed cones enable us to describe the local behavior of set-valued maps via various
generalized derivatives. All the set-valued mappings under consideration are supposed to have {\em locally closed graph} around the points in question.

\begin{definition}\label{DefGenDeriv}
Consider a  multifunction $F:\R^n\tto\R^m$ and let $(\xb,\yb)\in \gph F$.
\begin{enumerate}
\item[(i)]
 The multifunction $DF(\xb,\yb):\R^n\tto\R^m$ given by $\gph DF(\xb,\yb)=T_{\gph F}(\xb,\yb)$ is called the {\em graphical derivative} of $F$ at $(\xb,\yb)$.
\item[(ii)] The {\em outer limiting graphical derivative} of $F$ at $(\xb,\yb)$ is the multifunction
    $D^\sharp F(\xb,\yb):\R^n\tto\R^m$ given by
    \[\gph D^\sharp F(\xb,\yb)= T^\sharp_{\gph F}(\xb,\yb).\]

 \item[(iii)] The multifunction $D_*F(\xb,\yb):\R^n\tto\R^m$ given by $\gph D_*F(\xb,\yb)=T^P_{\gph F}(\xb,\yb)$ is called the {\em  strict (paratingent) derivative} of $F$ at $(\xb,\yb)$.
\item[(iv)]
 The multifunction $\widehat D^\ast F(\xb,\yb ): \R^m\tto\R^n$  defined by
 \[ \gph \widehat D^\ast F(\xb,\yb )=\{(y^*,x^*)\mv (x^*,-y^*)\in \widehat N_{\gph F}(\xb,\yb)\}\]
is called the {\em regular (Fr\'echet) coderivative} of $F$ at $(\xb,\yb )$.
\item [(v)]  The multifunction $D^\ast F(\xb,\yb ): \R^m \tto \R^n$,  defined by
 \[ \gph D^\ast F(\xb,\yb )=\{(y^*,x^*)\mv (x^*,-y^*)\in N_{\gph F}(\xb,\yb)\}\]
is called the {\em limiting (Mordukhovich) coderivative} of $F$ at $(\xb,\yb )$.
\end{enumerate}
\end{definition}
The outer limiting graphical derivative has been introduced by the authors in  \cite{GfrOut22a}.

If $F$ is single-valued, we can omit the second argument and write $DF(x)$, $\widehat D^*F(x),\ldots$ instead of $DF\big(x,F(x)\big)$, $\widehat D^*F\big(x,F(x)\big),\ldots$. However, be aware that when considering limiting objects at $x$ where $F$ is not continuous, it is not enough to consider only sequences $x_k\to x$ but we must work with sequences $\big(x_k,F(x_k)\big)\to\big(x,F(x)\big)$.

\begin{definition}
  Let $U\subset \R^n$ be open and consider a mapping  $F:U\to \R^m$. The {\em B-Jacobian} of $F$ at $x\in U$ is defined as
  \begin{equation}
    \overline{\nabla} F(x):=\{A\mv \exists x_k\to x: \mbox{$F$ is Fr\'echet differentiable at $x_k$ and }A=\lim_{k\to\infty}\nabla F(x_k)\}.
  \end{equation}
\end{definition}
Recall that the Clarke Generalized Jacobian is given by $\co\overline{\nabla} F(x)$, i.e., the convex hull of the B-Jacobian.

There exists the following relation between the B-Jacobian and the limiting coderivative of a single-valued mapping $F$, which states that every element from the B-Jacobian defines a certain subspace contained in the graph of the coderivative.
\begin{proposition}[{\cite[Proposition 2.4]{GfrOut22a}}]\label{PropBSubdiffCoderiv}
  Let $U\subset \R^n$ be open and let $F:U\to \R^m$ be a mapping.  Let $F$ be continuous at $x\in U$  and let $A\in \overline{\nabla} F(x)$. Then
  \[ (y^*,A^Ty^*)\in\gph D^*F(x)\ \forall y^*\in\R^m.\]
\end{proposition}
If the mapping $F:U\to\R^m$ is Lipschitz continuous, then by Rademacher's Theorem $F$ is differentiable almost everywhere in $U$ and $\norm{\nabla F(x)}$  is bounded there by the Lipschitz constant of $F$. Thus $\overline{\nabla} F(\xb)\not=\emptyset$ for Lipschitz continuous mappings $F$.

Let us now recall the following regularity notions.
\begin{definition}
  Consider a mapping $F:\R^n\tto\R^m$ and let $(\xb,\yb)\in\gph F$. Then
  \begin{enumerate}
    \item $F$ is said to be {\em metrically subregular at} $(\xb,\yb)$ if there exists $\kappa\geq 0$  along with some  neighborhood $U$ of $\xb$ such that
    \begin{equation}\label{EqSubreg}
      \dist{x,F^{-1}(\yb)}\leq \kappa\dist{\yb,F(x)}\ \forall x\in U.
    \end{equation}
  \item $F$ is said to be {\em strongly metrically subregular at} $(\xb,\yb)$ if there is $\kappa\geq 0$ together with some neighborhood $U$ of $\xb$ such that
\begin{equation}\label{EqStrongSubreg}
      \norm{x-\xb}\leq \kappa\dist{\yb,F(x)}\ \forall x\in U.
    \end{equation}
  \item $F$ is said to be {\em metrically regular around} $(\xb,\yb)$ if there is $\kappa\geq 0$ together with neighborhoods $U$ of $\xb$ and $V$ of $\yb$ such that
      \begin{equation}\label{EqMetrReg}
      \dist{x,F^{-1}(y)}\leq \kappa\dist{y,F(x)}\ \forall (x,y)\in U\times V.
    \end{equation}
  \end{enumerate}
\end{definition}
Note that conditions \eqref{EqStrongSubreg} implies that $F^{-1}(\yb)\cap U=\{\xb\}$.

Related with these regularity properties are the following Lipschitzian properties.
\begin{definition}
  Let $S:\R^m\tto\R^n$ be a mapping and let $(\yb,\xb)\in\gph S$. Then
  \begin{enumerate}
    \item $S$ is said to be {\em calm} at $(\yb,\xb)$ if there exists $\kappa\geq0$ along with a neighborhood $U$ of $\xb$ such that
    \[S(y)\cap U\subset S(\xb)+\kappa\norm{y-\yb}\B_{\R^n}\ \forall y\in\R^m.\]
    \item $S$ has the {\em isolated calmness} property at $(\yb,\xb)$ if
    there exists $\kappa\geq 0$ along with a neighborhood $U$ of $\xb$ such that
    \begin{equation}\label{EqIsolCalm}S(y)\cap U\subset \{\xb\}+\kappa\norm{y-\yb}\B_{\R^n}\ \forall y\in\R^m.
    \end{equation}
    \item $S$ has the {\em Aubin} property around $(\yb,\xb)$ if there is some constant $\kappa\geq0$ along with neighborhoods $V$ of $\yb$ and $U$ of $\xb$ such that
    \[S(y)\cap U\subset S(y')+\kappa\norm{y-y'}\B_{\R^n}\ \forall y,y'\in V.\]
  \end{enumerate}
\end{definition}
Now the condition \eqref{EqIsolCalm} defining isolated calmness ensures that $S(\yb)\cap U=\{\xb\}$.

It is well-known, see, e.g., \cite{DoRo14}, that the property of (strong) metric subregularity for $F$ at $(\xb,\yb)$ with constant $\kappa$ is equivalent with the property of {\em (isolated) calmness}  for $F^{-1}$ at $(\yb,\xb)$ with constant $\kappa$.
Further, $F$ is metrically regular around $(\xb,\yb)$ with constant $\kappa$ if and only if the inverse mapping $F^{-1}$ has the so-called {\em Aubin property} around $(\yb,\xb)$ with constant $\kappa$.

The properties of  metric regularity and strong metric subregularity are stable under Lipschitzian and calm  perturbations, respectively, cf. \cite{DoRo14}. Further note that the property of  metric regularity holds around all points belonging to the graph of $F$ sufficiently close to the reference point, whereas the property of (strong) metric subregularity is guaranteed to hold only at the reference point. This leads to the following definition.
\begin{definition}\label{DefLocStrSubReg}
\begin{enumerate}
\item
  We say that the mapping $F:\R^n\tto\R^m$ is {\em (strongly) metrically subregular around} $(\xb,\yb)\in\gph F$ if there is $\kappa\geq0$ and a neighborhood $W$ of $(\xb,\yb)$ such that $F$ is (strongly) metrically subregular with constant $\kappa$ at every point $(x,y)\in\gph F\cap W$.\\
  In this case we will also speak about {\em (strong) metric subregularity on a neighborhood}.
\item We say that the mapping $S:\R^m\tto\R^n$ is called {\em (isolatedly) calm around} $(\yb,\xb)\in\gph S$ if there is some constant $\kappa\geq 0$ along with some neighborhood $W$ of $(\yb,\xb)$ such that $S$ is isolatedly calm with constant $\kappa$ at every point $(y,x)\in \gph S\cap W$.\\
    In this case we will also speak about {\em (isolated) calmness on a neighborhood}.
\end{enumerate}
\end{definition}
The notion of (strong) metric subregularity on a neighborhood was introduced in \cite[Definition 2.8]{GfrOut22a}.
Due to the relation between (strong) metric subregularity of $F$ and (isolated) calmness of $F^{-1}$ we immediately obtain the following result.
\begin{lemma}
  Let $F:\R^n\tto\R^m$ be a mapping and let $(\xb,\yb)\in\gph F$. Then $F$ is (strongly) metrically subregular around $(\xb,\yb)$ if and only if $F^{-1}$ is isolatedly calm around $(\yb,\xb)$.
\end{lemma}
Note that every polyhedral multifunction, i.e., a mapping whose graph is the union of finitely many convex polyhedral sets, is both metrically subregular and calm around every point of its graph by Robinson's result \cite{Ro81}. In this paper, we will restrict our investigations to the properties of strong metric subregularity and isolated calmness on a neighborhood. Let us first have a closer look on Definition \ref{DefLocStrSubReg}.

The mapping $F:\R^n\tto\R^m$ is strongly metrically subregular around $(\xb,\yb)\in\gph F$ if and only if there is some $\kappa\geq0$ together with some neighborhood $W$ of $(\xb,\yb)$ such that for every $(x,y)\in\gph F\cap W$ there is some neighborhood $U_{xy}$ of $x$ with
\[\dist{x',F^{-1}(y)}\leq \kappa \dist{y, F(x')}\ \forall x'\in U_{xy}.\]
Note that the neighborhoods $U_{xy}$ depends both on $x$ amd $y$ and can be arbitrarily small.

Similarly, the mapping $S:\R^m\tto\R^n$ is isolatedly calm around $(\yb,\xb)\in\gph S$ if and only if there is some $\kappa\geq 0$  together with some neighborhood $W$ of $(\yb,\xb)$ such that for every $(y,x)\in\gph S\cap W$ there is some neighborhood $U_{yx}$ of $x$ with
\[S(y')\cap U_{yx}\subset \{x\}+\kappa \norm{y'-y}\B_{\R^n}\ \forall y'\in\R^m.\]

In this paper we will use the  following point-based characterizations of the above regularity properties.
\begin{theorem}\label{ThCharRegByDer}
    Consider a mapping $F:\R^n\tto\R^m$ and let $(\xb,\yb)\in\gph F$. Then
    \begin{enumerate}
    \item[(i)] (Levy-Rockafellar criterion, see, e.g., \cite[Theorem 4.1]{DoRo14}) $F$ is strongly metrically subregular at $(\xb,\yb)$ if and only if
      \begin{equation}\label{EqLevRoCrit}
        0\in DF(\xb,\yb)(u)\ \Rightarrow u=0.
      \end{equation}
    \item[(ii)] (Mordukhovich criterion, see, e.g., \cite[Theorem 3.3]{Mo18}) $F$ is metrically regular around $(\xb,\yb)$ if and only if
      \begin{equation}
            \label{EqMoCrit} 0\in D^*F(\xb,\yb)(y^*)\ \Rightarrow\ y^*=0.
      \end{equation}
      \item[(iii)]$F$ is strongly metrically subregular around $(\xb,\yb)$ if and only if
        \begin{equation}\label{EqLocStrSubregCrit}
          0\in D^\sharp F(\xb,\yb)(u)\ \Rightarrow u=0.
        \end{equation}
    \end{enumerate}
\end{theorem}
The characterization (iii) of strong metric subregularity on a neighborhood was shown in \cite[Theorem 6.1]{GfrOut22a} for the special case $m=n$. But a close inspection of the proof of
\cite[Theorem 6.1]{GfrOut22a} shows that it can be used without any modification to show the general case as well.

Next we  introduce the \ssstar sets and mappings.

\begin{definition}\label{DefSemiSmooth}
\begin{enumerate}
\item  A set $A\subseteq\R^s$ is called {\em \ssstar} at a point $\xb\in A$ if for every $\epsilon>0$ there is some $\delta>0$ such that
\[\vert \skalp{x^*,x-\xb}\vert\leq \epsilon\norm{x-\xb}\norm{x^*}\]
holds for all $x\in A\cap\B_\delta(\xb)$ and all $x^*\in \widehat N_A(x)$.
\item A set-valued mapping $F:\R^n\tto\R^m$ is called {\em \ssstar} at a point $(\xb,\yb)\in\gph F$, if
$\gph F$ is \ssstar at $(\xb,\yb)$, i.e., for every $\epsilon>0$ there is some $\delta>0$ such that
\[\vert \skalp{x^*,x-\xb}-\skalp{y^*,y-\yb}\vert\leq \epsilon\norm{(x,y)-(\xb,\yb)}\norm{(x^*,y^*)}\]
holds for all $(x,y)\in\gph F\cap\B_\delta(\xb,\yb)$ and all $(y^*,x^*)\in \gph\widehat D^*F(x,y)$.
\end{enumerate}
\end{definition}
Note that the above definitions of \ssstar sets and multifunctions are not the same as the ones introduced in \cite{GfrOut21}, but by \cite[Proposition 3.2, Corollary 3.3]{GfrOut21} they are equivalent.

The class of semismooth* sets and mappings is rather broad.
\begin{proposition}  \label{PropSSstar}
\begin{enumerate}
  \item Any closed convex set $A \subset \mathbb{R}^{s}$  is \ssstar at each $\bar{x} \in A$.
  \item Assume that we are given closed sets $A_i\subset\R^s$, $i=1,\ldots p$, and $\xb\in
  A:=\bigcup_{i=1}^p A_i$. If the sets $A_i$, $i\in \bar I:=\{j\mv\xb\in A_j\}$, are \ssstar at
  $\xb$, then so is the set $A$.
  \item Every closed subanalytic set $A$ is \ssstar at each $\xb\in A$.
\end{enumerate}
\end{proposition}
The first two statements of this proposition can be found in \cite[Propsition 3.4, Proposition 3.5]{GfrOut21}, whereas the last statement follows from \cite[Theorem 2]{Jou07}.

We now state a sufficient condition for the \ssstar property of sets with constraint structure.
\begin{proposition}\label{PropSSstarNew}Let $A=\{x\in\R^s\mv \Phi(x)\in D\}$, where $\Phi:\R^s\to\R^p$ is continuously differentiable and $D\subset \R^p$ is a closed set. Given $\xb\in A$, assume that the mapping $x\mapsto F(x):=\Phi(x)-D$ is metrically subregular at $(\xb,0)$ and assume that $D$ is \ssstar at $\Phi(\xb)$. Then $A$ is \ssstar at $\xb$.
\end{proposition}
\begin{proof}
  By metric subregularity of $F$ there exists a real $\kappa> 0$ together with some open neighborhood $U$ such that \eqref{EqSubreg} holds. It follows that for  every $x\in A\cap U$ the mapping $F$ is metrically subregular with constant $\kappa$ at $(x,0)$ and thus, by \cite[Theorem 3]{GfrYe17} there holds
  \[N_A(x)\subset\{\nabla \Phi(x)^Ty^*\mv y^*\in N_D\big(\Phi(x)\big)\cap\kappa\norm{x^*}\B_{R^p}\},\ x\in A\cap U. \]
 Since $D$ is \ssstar at $\Phi(\xb)$, by \cite[Proposition 3.2]{GfrOut21} there is some radius $\rho>0$ such that
  \[\vert \skalp{y^*,d-\Phi(\xb)}\vert\leq \frac{\epsilon}{2L\kappa}\norm{d-\Phi(\xb)}\norm{y^*}\ \forall d\in D\cap \B_\rho\big(\Phi(\xb)\big)\ \forall y^*\in N_D(d),\]
  where $L$ denotes the Lipschitz constant of $\Phi$ on some ball $B_\delta(\xb)\subset U$. Next choose $0<\bar\delta<\min\{\delta,\rho/L\}$ such that
  \[\norm{\Phi(\xb)-\Phi(x)-\nabla\Phi(x)(\xb-x)}\leq \frac\epsilon{2\kappa}\norm{x-\xb},\ x\in\B_{\bar\delta}(\xb),\]
   and consider $x\in A\cap \B_{\bar\delta}(\xb)$ and $x^*\in N_A(x)$ together with $y^*\in N_D\big(\Phi(x)\big)$ satisfying $\norm{y^*}\leq\kappa\norm{x^*}$ and $x^*=\nabla\Phi(x)^Ty^*$.
   Then
   \begin{align*}
     \vert \skalp{x^*,x-\xb}\vert&=\vert\skalp{y^*,\nabla\Phi(x)(x-\xb)}\vert\\
     &\leq \vert\skalp{y^*,\Phi(x)-\Phi(\xb)}\vert+\vert\skalp{y^*,\Phi(\xb)-\Phi(x)-\nabla\Phi(x)(x-\xb)}\vert\\
     &\leq \frac{\epsilon}{2L\kappa}\norm{\Phi(x)-\Phi(\xb)}\norm{y^*}+\norm{y^*}\norm{\Phi(\xb)-\Phi(x)-\nabla\Phi(x)(\xb-x)}\\
     &\leq \frac{\epsilon}{2L\kappa}L\norm{x-\xb}\kappa\norm{x^*}+\frac \epsilon{2\kappa}\kappa\norm{x^*}\norm{x-\xb}=\epsilon\norm{x-\xb}\norm{x^*},
   \end{align*}
   verifying that $A$ is \ssstar at $\xb$.
\end{proof}

In case of single-valued Lipschitzian mappings the \ssstar property is equivalent with the semismooth property introduced by Gowda \cite{Gow04}, which is weaker than the one in \cite{QiSun93}.

\section{Preliminaries}
This section is composed from two parts. The first one, Section 3.1, contains a generalization of the basic facts about the SCD mappings from \cite[Section 3]{GfrOut22a} to multifunctions between different finite-dimensional spaces. Section 3.2 is then devoted to the important notion of SCD regularity, playing a crucial role in the subsequent development.

\subsection{SCD mappings}
Let us denote by $\Z_{nm}$ the metric space of all n-dimensional subspaces of $\mathbb{R}^{n+m}$ equipped with the metric
\begin{equation}\label{eq-30}
d_{\Z_{nm}}(L_{1},L_{2}):=\|P_{1}-P_{2}\|,
\end{equation}
where $P_{i}$ is the symmetric $(n+m) \times (n+m)$ matrix representing the orthogonal projection onto $L_{i}, i=1,2$.
\if{Given a subspace $L \in \Z_{nm}$, we denote by $\mathcal{M}_{nm}$ the collection of all $(n+m)\times n$ matrices $Z$ having the full rank $n$, and put
 %%%
\[
\begin{split}
& \mathcal{M}(L):= \{Z \in \mathcal{M}_{nm} \mv \rge Z = L\},\\
& \mathcal{M}^{\orth}(L):=\{Z \in \mathcal{M}(L)\mv Z^{T}Z = I\}.
\end{split}
\]
%%%
It follows that if $Z \in \mathcal{M}^{\orth}(L)$, then the columns of $Z$ create an orthogonal basis for $L$. Similarly as in the case $m=n$ the orthogonal projection onto $L$ admits the representations
%%%
\begin{equation}\label{eq-31}
P=Z(Z^{T}Z)^{-1} Z^{T} \mbox{ for } Z \in \mathcal{M}(L)
\end{equation}
%%%
and
%%%
\begin{equation}\label{eq-32}
P=ZZ^{T} \mbox{ for } Z \in \mathcal{M}^{\orth}(L).
\end{equation}
%%%
Of course, $L$ can be represented as the   range space of a much larger class of matrices than $\mathcal{M}(L)$. They have more columns than $n$ but their column rank amounts to $n$.
}\fi
Throughout  the whole paper we make use of the following relationships.
\begin{lemma}\label{LemBasic}

\begin{enumerate}
\if{
 \item [(i)]
 Let $Z_{k}\in \mathcal{M}_{nm}$ be a sequence converging to some $Z \in \mathcal{M}_{nm}$. Then the subspaces $\rge Z_{k}$ converge in $\Z_{nm}$ to the subspace $\rge Z \in \Z_{nm}$.
 \item [(ii)]
 Let $L_{k}\in \Z_{nm}$ be a sequence converging to some $L \in \Z_{nm}$. Then there is a sequence $Z_{k}\in \red{\M(L_k)}$ %\mathcal{M}_{nm}
  converging to some $Z \in \mathcal{M}(L)$.
  }\fi
 \item [(i)]
 Let $A_{k}$ be a sequence of $(n+m) \times (n+l)$ full-column-rank matrices converging to a full-column-rank matrix $A$ and let $L_{k}\in \Z_{nl}$ be a sequence of subspaces converging to $L \in \Z_{nl}$. Then $\lim\limits_{k \rightarrow \infty} d_{\Z_{nm}}(A_{k}L_{k}, AL)=0$.
 \item [(ii)]
 The metric space $\Z_{nm}$ is (sequentially) compact.
\end{enumerate}
\end{lemma}
The above statements can be proved in the same way as  their counterparts in \cite[Lemma 3.1(iii),(iv)]{GfrOut22a} and therefore the proofs are omitted.

To be consistent with the  notation in \cite{GfrOut22a} we will write $\Z_{n}$ instead of $\Z_{nn}$.

 With  each $L \in \Z_{nm}$ one can associate its {\em adjoint} subspace $L^{*}$ defined by
  \begin{equation}\label{eq-31a}
L^{*}:=\{(-v^{*},u^{*}) \in \mathbb{R}^{m} \times \R^{n} \mv (u^{*},v^{*})\in L^{\perp}\}.
  \end{equation}
  Since $\dim L^{\perp}=m$, it follows that $L^{*}\in \Z_{mn}$ (i.e., its dimension is $m$). It is easy to see that
\begin{equation}\label{eq-34}
L^{*}= S_{nm} L^{\perp}, \mbox{ where } S_{nm} =
\left( \begin{array}{lc}
0 & -I_{m}\\
I_{n} & 0
\end{array}\right),
\end{equation}
yielding $(L^{*})^\perp=\{z\mv S_{nm}^Tz\in (L^\perp)^\perp=L\}$. Hence we obtain
\begin{align}\nonumber(L^*)^*&=\{(-u,v)\in\R^n\times\R^m\mv (v,u)\in (L^*)^\perp\}= \{(-u,v)\in\R^n\times\R^m\mv S_{nm}^T(v,u)=(u,-v)\in L\}\\
\label{EqIsometry}&=-L=L.
\end{align}
Further, if we denote by $P_{L^*}$, $P_{L^\perp}$ and $P_L$ the symmetric $(n+m)\times (n+m)$ matrices representing the orthogonal projections onto $L^*$, $L^\perp$ and $L$, respectively, then we have $P_{L^\perp}=I_{n+m}-P_L$ and, since $S_{nm}$ is orthogonal,
\[P_{L^*}=S_{nm}P_{L^\perp}S_{nm}^T=I_{n+m}-S_{nm}P_L S_{nm}^T.\]
We conclude that for any two subspaces $L_1,L_2\in \Z_{nm}$ there holds
\[d_{\Z_{mn}}(L_1^*,L_2^*)=\norm{I_{n+m}-S_{nm}P_{L_1} S_{nm}^T-(I_{n+m}-S_{nm}P_{L_2} S_{nm}^T)}=\norm{P_{L_1}-P_{L_2}}=d_{\Z_{n,m}}(L_1,L_2)\]
and thus the mapping $L\to L^*$ is an isometry between $\Z_{nm}$ and $\Z_{mn}$.
%%%
In what follows, the symbol $L^*$ signifies both the adjoint subspace to some $L\in \Z_{nm}$ as well as an arbitrary subspace from $\Z_{mn}$. This double role, however, cannot lead to a confusion.

Consider now a mapping $F: \mathbb{R}^{n}  \rightrightarrows \mathbb{R}^{m}$.

\begin{definition}
We say that $F$ is {\em graphically smooth of dimension} $n$ at $(\bar{x},\bar{y})$ if $T_{\gph F}(\bar{x},\bar{y})\in \Z_{nm}$. By $\mathcal{O}_{F}$
 we denote the subset of $\gph F$, where $F$ is graphically smooth of dimension $n$.\\
\end{definition}
Clearly, for $(x,y)\in \mathcal{O}_{F}$ and $L=T_{\gph F}(x,y)=\gph DF(x,y)$ it holds that
$L^{\perp}=\hat{N}_{\gph F}(x,y)$ and $L^{*}=\gph \widehat{D}^{*}F(x,y)$.\\

Next we introduce the four derivative-like mappings $\widehat{\mathcal{S}}F:\mathbb{R}^{n} \times  \mathbb{R}^{m} \tto \Z_{nm},
\widehat{\mathcal{S}^{*}}F:\mathbb{R}^{n} \times  \mathbb{R}^{m} \tto \Z_{mn},
{\mathcal{S}}F:\mathbb{R}^{n} \times  \mathbb{R}^{m} \rightrightarrows \Z_{nm}$ and ${\mathcal{S}}^{*}F:\mathbb{R}^{n} \times  \mathbb{R}^{m} \rightrightarrows \Z_{mn}$ defined by
\begin{gather*}
\widehat{\mathcal{S}}F(x,y):=\begin{cases}\gph DF(x,y)&\mbox{ if }(x,y)\in \mathcal{O}_{F}\\
\emptyset& \mbox{ otherwise, }
\end{cases}\\
%%%
\widehat{\mathcal{S}^{*}}F(x,y):=\begin{cases}
\gph \widehat{D}^{*}F(x,y)&\mbox{ if }(x,y)\in \mathcal{O}_{F}\\
\emptyset&\mbox{ otherwise, }
\end{cases}
\end{gather*}
%%%
\begin{align*}
{\mathcal{S}}F(x,y):&= \Limsup\limits_{\stackrel{\gph F}{(u,v)\rightarrow(x,y)}}
\widehat{\mathcal{S}}F(u,v)\\
&=\{L \in \Z_{nm} \mv \exists(x_{k},y_{k})\stackrel{\mathcal{O}_{F}}{\rightarrow}
(x,y) \mbox{ such that } \lim d_{\Z_{nm}}\big(L, \gph DF(x_{k},y_{k})\big)=0\},
\end{align*}
and
\begin{align*}
{\mathcal{S}^{*}}F(x,y):&= \Limsup\limits_{\stackrel{\gph F}{(u,v)\rightarrow(x,y)}}
\widehat{\mathcal{S}^{*}}F(u,v)
\\
&=\{L^{*} \in \Z_{mn} \mv \exists(x_{k},y_{k})\stackrel{\mathcal{O}_{F}}{\rightarrow}
(x,y) \mbox{ such that } \lim d_{\Z_{mn}}\big(L^{*}, \gph \widehat{D}^{*}F(x_{k},y_{k})\big)=0\}.
\end{align*}
Both $\Sp F$ and $\Sp^*F$ constitute  generalized derivatives of $F$ whose elements, by virtue of the above definitions, are subspaces of the graphs of the outer limiting graphical derivative and the limiting coderivative:
\begin{gather}
\label{EqSpSubsetDF}  L\subset\gph D^\sharp F(x,y)\subset \gph D_{*}F(x,y)\ \forall L\in\Sp F(x,y),\\
\label{EqSp*SubsetDF*}L^*\subset \gph D^{*}F(x,y)\ \forall L^*\in\Sp^*F(x,y).
\end{gather}
In what follows  ${\mathcal{S}}F$ will be called {\em SC (subspace containing) limiting graphical derivative} and ${\mathcal{S}^{*}}F$ will be termed {\em SC limiting coderivative at $(x,y)$}.
\if{ By virtue of the above definitions one has that
\[
{\mathcal{S}^{*}}F(x,y)\subset \gph D^{*}F(x,y)\big(\mbox{ and }  {\mathcal{S}}F(x,y)\subset                 \gph D_{*}F(x,y)\big).
\]}\fi

\if{
Analogously to the distance in $\Z_{nm}$ we may introduce a distance (denoted by $d^\perp$) also in the space of orthogonal subspaces to elements $L \in \Z_{nm}$ via
\[
d^{\perp}(L^{\perp}_{1}, L^{\perp}_{2}):= \| P^{\perp}_{1} - P^{\perp}_{2}\|,
\]
where $P^{\perp}_{1}$ stands for the $(n+m)\times (n+m)$ matrix of orthogonal projection onto $L^{\perp}_{i}, i=1,2$. Since
\[
P_{i}+P^{\perp}_{i}=I, \quad i = 1,2,
\]
it follows that
\[
d^{\perp}(L^{\perp}_{1}, L^{\perp}_{2})= \| I-P_{1}-(I-P_{2})\| = \| P_{1}-P_{2}\| = d_{\Z_{nm}}(L_{1}, L_{2}).
\]
}\fi
Due to the isometry $L\to L^*$ %this isometry and relation (\ref{eq-34})
 we obtain a useful mutual relationship between $\mathcal{S}F(\bar{x},\bar{y})$ and $\Sp^*F(\xb,\yb)$. 
It holds, namely, that
%%%
\begin{equation}\label{eq-35}
\Sp^*%\varphi^{*}
F(\bar{x},\bar{y})=\{L^{*}\mv L \in \mathcal{S}F(\bar{x},\bar{y})\} ~\mbox{ and } ~ \mathcal{S}F(\bar{x},\bar{y}) = \{L \mv L^{*}\in \mathcal{S}^{*}F(\bar{x},\bar{y}) \},
\end{equation}
which enables us together with (\ref{eq-34}) a simple conversion of the statements in terms of $L \in \mathcal{S}F(\bar{x},\bar{y})$ to statements in terms of $ L^{*}\in \mathcal{S}^{*}F(\bar{x},\bar{y}) $ and vice versa.

On the basis of $\Sp^*F(\bar{x},\bar{y})$ we may now introduce the following notion playing a crucial role in the sequel. \\
\begin{definition}
A mapping
$F:\mathbb{R}^{n} \rightrightarrows   \mathbb{R}^{m}$ is said to have the {\em SCD property} at $(\bar{x},\bar{y})\in \gph F$, provided $\mathcal{S}^{*}F(\bar{x},\bar{y})\neq \emptyset$. $F$ is termed an {\em SCD mapping }if it has the SCD property at all points of $\gph F$.
\end{definition}
By virtue of \eqref{eq-35}, the \SCD property at $(\xb,\yb)$ is obviously equivalent with the condition $\Sp F(\xb,\yb)\not=\emptyset$.

Since we consider convergence in the compact metric space $\Z_{nm}$, we obtain readily the following result.
\begin{lemma}[{cf.\cite[Lemma 3.6]{GfrOut22a}}]\label{LemSCDproperty}
 A mapping $F:\R^n\tto\R^m$ has the \SCD property at $(x,y)\in\gph F$ if and only if $(x,y)\in\cl \OO_F$. Further, $F$ is an \SCD mapping if and only if $\cl \OO_F=\cl \gph F$, i.e., $F$ is graphically smooth of dimension $n$ at the points of a dense subset of its graph.
\end{lemma}

The derivatives $\Sp F$ and $\Sp^*F$ can be considered as a generalization of the B-Jacobian to multifunctions. In case of single-valued continuous mappings one has the following relationship.
\begin{lemma}\label{LemSCDSingleValued}
  Let $U\subset \R^n$ be open and let $f:U\to\R^m$ be continuous. Then for every $x\in U$ there holds
  \begin{align}
    \label{EqGenBSubdiff1}&\Sp f(x):=\Sp\big(x,f(x)\big)\supseteq \{\rge(I,A)\mv A\in\overline{\nabla} f(x)\},\\
    \label{EqGenBSubdiff2}&\Sp^* f(x):=\Sp^*\big(x,f(x)\big)\supseteq \{\rge(I,A^T)\mv A\in\overline{\nabla} f(x)\}.
  \end{align}
  If $f$ is Lipschitz continuous near $x$, these inclusions hold with equality and $f$ has the \SCD property around $x$.
\end{lemma}
\begin{proof}
  We can carry over the proof of \cite[Lemma 3.11]{GfrOut22a} with marginal modifications.
\end{proof}

%The class of mappings having one of the above properties is rather large as we will show later.

\subsection{SCD regularity}

For $m=n$ we recall the following weakening of metric regularity tailored to SCD mappings.\\
%%%%
\begin{definition}[{\cite[Definition 4.1]{GfrOut22a}}]
A mapping $F:\mathbb{R}^{n} \rightrightarrows  \mathbb{R}^{n}$ is called {\em SCD regular} around $(\bar{x},\bar{y})$, provided it has the SCD property on a neighborhood of $(\bar{x},\bar{y})$ and for all $L^{*}\in \mathcal{S}^{*}F(\bar{x},\bar{y}) $ one has the implication
\begin{equation}\label{eq-40}
(v^*,0)\in L^*\ \Rightarrow\ v^*=0. %(y^{*}, 0) \in L^{*} \Rightarrow u^{*} = 0.
\end{equation}
\end{definition}
It is easy to see, cf. \cite[Lemma 4.5]{GfrOut22a}, that implication (\ref{eq-40})
 is equivalent with the requirement that
 \begin{equation}\label{eq-41}
(u,0)\in L\ \Rightarrow\ u = 0 ~~ \mbox{ for all } ~~ L \in \mathcal{S}F(\bar{x},\bar{y}).
 \end{equation}
%%%
Further we observe that SCD regularity persists on a neighborhood of $(\bar{x},\bar{y})$, cf. \cite[Proposition 4.8]{GfrOut22a}, and, taking into account \eqref{EqSp*SubsetDF*} and the Mordukhovich criterion, it is implied by the (classical) metric regularity of $F$ around $(\bar{x},\bar{y})$.

The main vehicle in our stability analysis of SCD mappings in the fifth section are the following statements taken over from  \cite[Theorem 6.2, Corollary 6.4]{GfrOut22a}.

\begin{theorem}\label{ThStrMetrSubreg}
Assume that $F:\mathbb{R}^{n} \rightrightarrows  \mathbb{R}^{n}$ is SCD regular around a point $(\bar{x},\bar{y})\in \gph F$. Then there is a neighborhood $\mathcal{U}$ of $(\bar{x},\bar{y})$ such that $F$ is strongly metrically subregular at each point of $\gph F\cap \mathcal{U}$, where $F$ is \ssstar.
\end{theorem}
\begin{corollary}\label{CorStrMetrSubreg}
  Assume that $F:\R^n\tto\R^n$ is \ssstar and has the SCD property  around $(\xb,\yb)\in\gph F$. Then $F$ is strongly metrically subregular around $(\xb,\yb)$ if and only if $F$ is SCD regular around $(\xb,\yb)$.
\end{corollary}
 Conversely, thanks to Theorem \ref{ThCharRegByDer}(iii) and \eqref{EqSpSubsetDF}, strong metric subregularity around $(\bar{x},\bar{y})$ implies the SCD regularity at $(\bar{x},\bar{y})$  even in absence of the \ssstar property. Since by virtue of \cite[Theorem 3H.3]{DoRo14} $F$ is strongly metrically subregular at
$(\bar{x},\bar{y})$  if and only if $F^{-1}$ is isolatedly calm at $(\yb,\xb)$, %$(\bar{x},\bar{y})$,
Corollary \ref{CorStrMetrSubreg} thus provides us with a workable characterization of isolated calmness of inverses to SCD mappings having the \ssstar property.

Let us compare Corollary \ref{CorStrMetrSubreg} with the characterization of strong metric subregularity on a neighborhood provided by Theorem \ref{ThCharRegByDer}(iii). To this aim we write down relation \eqref{EqLocStrSubregCrit} equivalently in the form
\begin{equation}\label{EqLocStrSubregCrit1}
  (u,0)\in\gph D^\sharp F(\xb,\yb)\ \Rightarrow\ u=0.
\end{equation}
By taking into account \eqref{EqSpSubsetDF} and \eqref{eq-41}, we see that we need not to check \eqref{EqLocStrSubregCrit1} for the whole graph of $D^\sharp F(\xb,\yb)$, but only for the part which is given by the subspaces contained in $\Sp F(\xb,\yb)$. It seems that for the analysis of strong metric subregularity and isolated calmness on a neighborhood of \ssstar SCD mappings the outer limiting graphical derivative is much too large and contains useless parts. Moreover, it seems that the outer limiting graphical derivative is much harder to compute than the SC limiting graphical derivative.

Because of the mentioned relationship between the metric regularity and SCD regularity and Theorem \ref{ThStrMetrSubreg} we arrive finally at the following corollary.
\begin{corollary}
Assume that an SCD mapping $F:\mathbb{R}^{n} \rightrightarrows  \mathbb{R}^{n}$ is metrically regular and \ssstar around $(\bar{x},\bar{y})$. Then $F^{-1}$  has not only the Aubin property around $(\yb,\xb)$, %$(\bar{x},\bar{y})$,
 but it is also isolatedly calm around $(\yb,\xb)$. %$(\bar{x},\bar{y})$.
\end{corollary}

\section{Calculus}
In this section we present some calculus rules for SCD mappings which can be useful in various situations.

Consider a mapping $F:\mathbb{R}^{n} \rightrightarrows  \mathbb{R}^{m}$ defined by
\begin{equation}\label{EqImplF}
\gph F = \{(x,y)\in \mathbb{R}^{n} \times  \mathbb{R}^{m} \mv \Phi (x,y)\in \gph Q\},
\end{equation}
where $\Phi: \R^n\times\R^m\to\R^l\times\R^m$ is a continuously differentiable function and $Q:\R^l\tto\R^m$ is a closed-graph mapping.
\begin{theorem}\label{ThCalc1}
  Assume that $(\xb,\yb)\in \gph F$, $Q$ has the SCD property at $\Phi(\xb,\yb)$ and the $(l+m) \times(n+m)$ matrix $\nabla \Phi(\bar{x},\bar{y})$ has full row rank $l+m$. Then $F$ has the SCD property at $(\bar{x},\bar{y})$,
\begin{equation}\label{eq-46}
\Sp F(\bar{x},\bar{y})=\{L \in \Z_{nm}\mv \nabla\Phi (\bar{x},\bar{y})L \in \Sp Q\big(\Phi(\bar{x},\bar{y})\big)\}
\end{equation}
and
\begin{equation}\label{eq-47}
\Sp^{*}F (\bar{x},\bar{y})=
\{L^{*} \in \Z_{mn} \mv L^{*} = S_{nm}\nabla\Phi
(\bar{x},\bar{y})^{T}S^{T}_{lm}M^* \mbox{ with } M^* \in \Sp^{*} Q \big(\Phi(\bar{x},\bar{y})\big)\}.
\end{equation}
\end{theorem}
\begin{proof}
Since $\nabla\Phi(\xb,\yb)$ is surjective, the mapping $\Phi$ is metrically regular around $\big((\xb,\yb),\Phi(\xb,\yb)\big)$, cf. \cite[Example 9.44]{RoWe98}. Moreover, there is an open neighborhood $\mathcal{W}$ of $(\xb,\yb)$ such that $\nabla\Phi(x,y)$ is surjective for all $(x,y)\in \mathcal{W}$ and $\widetilde{\mathcal{W}}=\Phi(\mathcal{W})$ is open. By virtue of \cite[Exercise 6.7]{RoWe98} it holds that
\begin{equation}\label{EqTangCone}
T_{\gph F}(x,y)=\{w\in  \R^n\times\R^m \mv\nabla\Phi(x,y)w \in T_{\gph Q}\big(\Phi(x,y)\big) \}
\end{equation}
for all $(x,y) \in \gph F \cap \mathcal{W}$. We now claim that
\begin{equation}\label{EqClaim1}
  \OO_Q\cap\widetilde{\mathcal{W}}=\{\Phi(x,y)\mv (x,y)\in \OO_F\cap\mathcal{W}\}.
\end{equation}
Indeed, consider $(x,y)\in \OO_F\cap\mathcal{W}$ and take two tangents $q_1,q_2\in T_{\gph Q}\big(\Phi(x,y)\big)$. Since $\nabla\Phi(x,y)$ is surjective, there exist $w_i$, $i=1,2$, with $\nabla \Phi(x,y)w_i=q_i$ implying $w_i\in T_{\gph F}(x,y)$ by \eqref{EqTangCone}. Since $T_{\gph F}(x,y)$ is a subspace, we have $\alpha_1w_1+\alpha_2w_2\in T_{\gph F}(x,y)$ $\forall\alpha_1,\alpha_2\in\R$ and consequently
\[\nabla\Phi(x,y)(\alpha_1w_1+\alpha_2w_2)=\alpha_1q_1+\alpha_2q_2\in T_{\gph Q}\big(\Phi(x,y)\big).\]
Hence $T_{\gph Q}\big(\Phi(x,y)\big)$ is a subspace.
From $(x,y)\in \OO_F$ we deduce that the dimension of the subspace $T_{\gph F}(x,y)$ is $n$. On the other hand, by \eqref{EqTangCone} together with the surjectivity of $\nabla \Phi(x,y)$, the dimension of $T_{\gph F}(x,y)$ equals to the dimension of the subspace $T_{\gph Q}\big(\Phi(x,y)\big)$ plus $(n+m)-(k+m)$, the dimension of the nullspace of $\nabla\Phi(x,y)$. Hence, the dimension of $T_{\gph Q}\big(\Phi(x,y)\big)$ is $k$ and $\Phi(x,y)\in\OO_Q\cap \widetilde{\mathcal{W}}$ is verified.

Next, consider $z\in \OO_Q\cap\widetilde{\mathcal{W}}$. Then we can find $(x,y)\in\mathcal{W}$ such that $z=\Phi(x,y)$ and using similar arguments as above, we can show that $T_{\gph F}(x,y)$ is a subspace of dimension $n$ implying $(x,y)\in (x,y)\in \OO_F\cap\mathcal{W}$. Hence our claim \eqref{EqClaim1} holds true.

Since $Q$ has the \SCD property at $\Phi(\xb,\yb)$, there holds $\Sp Q\big(\Phi(\xb,\yb)\big)\not =\emptyset$. Consider $M\in \Sp Q\big(\Phi(\xb,\yb)\big)$ together with a sequence $z_k\longsetto{\OO_Q}\Phi(\xb,\yb)$   such that $M_k:=T_{\gph Q}(z_k)\longsetto{\Z_{lm}} M$. For every $k$ sufficiently large we can find $(x_k,y_k)\in \mathcal{W}$ with $z_k=\Phi(x_k,y_k)$ and, due to the metric regularity of $\Phi$, $(x_k,y_k)\to(\xb,\yb)$. Further, $M_k^\perp$ converges in $\Z_{ml}$ to $M^\perp$. Let $L_k:=T_{\gph F}(x_k,y_k)=\nabla\Phi(x_k,y_k)^{-1}M_k$. Here, $\nabla\Phi(x_k,y_k)^{-1}$ denotes the inverse of the linear mapping induced by $\nabla\Phi(x_k,y_k)$. By our claim \eqref{EqClaim1} we have that $L_k\in \Z_{nm}$ and, since $L_k^\perp=\nabla\Phi(x_k,y_k)^TM_k^\perp$ by \cite[Corollary 16.3.2]{Ro70}, $L_k^*=S_{nm}\nabla\Phi(x_k,y_k)^TM_k^\perp$ converges to $L^*:=S_{nm}\nabla\Phi(\xb,\yb)^TM^\perp=S_{nm}\nabla\Phi(\xb,\yb)^TS_{lm}^TM^*$ by Lemma \ref{LemBasic}(i). On the other hand, since $\nabla\Phi(\xb,\yb)^TM^\perp=\big(\nabla\Phi(\xb,\yb)^{-1}M)^\perp$, we obtain $L=\nabla\Phi(\xb,\yb)^{-1}M$. These arguments show the inclusion ''$\supset$'' in \eqref{eq-46} and \eqref{eq-47}.

In order to show the reverse inclusion, consider $L\in \Sp F(\xb,\yb)$ together with sequences $(x_k,y_k)\longsetto{\OO_F}(\xb,\yb)$ and $L_k:=T_{\gph F}(x_k,y_k)\longsetto{\Z_{nm}}L$. By \eqref{EqTangCone} and \eqref{EqClaim1} together with the surjectivity of $\nabla\Phi(x_k,y_k)$, we obtain that $M_k:=\nabla\Phi(x_k,y_k)L_k=T_{\gph Q}\big(\Phi(x_k,y_k)\big)\in \Z_{lm}$. The metric space $\Z_{lm}$ is compact and thus, after possibly passing to a subsequence, we may assume that $M_k$ converges in $\Z_{lm}$ to some $M\in\Sp Q\big(\Phi(\xb,\yb)\big)$. Utilizing the same arguments as before, we obtain that the sequence
\[L_k^*=S_{nm}\nabla\Phi(x_k,y_k)^TS_{lm}^TM_k^*\]
converges to $L^*=S_{nm}\nabla\Phi(\xb,\yb)^TS_{lm}^TM^*$ and $L=\nabla\Phi(\xb,\yb)^{-1}M$. This completes the proof.
\end{proof}

As a first consequence of this theorem we derive that graphically Lipschitzian mappings have the SCD property.
\begin{definition}[cf.{\cite[Definition 9.66]{RoWe98}}]\label{DefGraphLip}A mapping $F:\R^n\tto\R^m$ is {\em graphically Lipschitzian of dimension $d$} at $(\xb,\yb)\in\gph F$ if there is an open neighborhood $W$ of $(\xb,\yb)$ and a one-to-one mapping $\Phi$ from $W$ onto an open subset of $\R^{n+m}$ with $\Phi$ and $\Phi^{-1}$ continuously differentiable, such that $\Phi(\gph F\cap W)$  is the graph of a Lipschitz continuous mapping $f:U\to\R^{n+m-d}$, where $U$ is an open set in $\R^d$.
\end{definition}
Many mappings $F:\R^n\tto\R^n$, important in applications, are graphically Lischitzian of dimension $n$. As an example we mention the subdifferential mapping of prox-regular and subdifferentially continuous functions $f:\R^n\to\overline{\R}$, cf. \cite[Proposition 13.46]{RoWe98}.
\begin{corollary}
  Assume that $F:\R^n\tto\R^m$ is graphically Lipschitzian of dimension $n$ at $(\xb,\yb)\in\gph F$. Then $F$ has the SCD property at $(\xb,\yb)$.
\end{corollary}
\begin{proof}
  Let $\Phi$, $W$, $U$ and $f$ be as in Definition \ref{DefGraphLip} and observe that $\gph F\cap W=\{(x,y)\mv \Phi(x,y)\in \gph Q\}$, where
  \[Q(u):=\begin{cases}
    \{f(u)\}&\mbox{if $u\in U$},\\
    \emptyset&\mbox{else.}
  \end{cases}\]
  By Lemma \ref{LemSCDSingleValued}, $Q$ has the SCD property at $\big(\bar u,f(\bar u)\big):=\Phi(\xb,\yb)$ and the statement follows from Theorem \ref{ThCalc1}.
\end{proof}
Let us now provide a calculus rule for the outer limiting tangent cone.
\begin{proposition}\label{PropCalcTsharp}
  Let $\Phi:\R^n\to\R^m$ be continuously differentiable, let $A\subset\R^m$ be a closed set and consider
  \[C:=\{x\in\R^n \mv \Phi(x)\in A\}.\]
  Then for any $\xb\in C$ there holds
  \begin{equation}\label{EqCalcTsharp}
    T^\sharp_C(\xb)\subset\{u\mv \nabla \Phi(\xb)u\in T^\sharp_A\big(\Phi(\xb)\big)\}
  \end{equation}
  If $\nabla \Phi(\xb)$ has full row rank $m$ then this inclusion holds with equality.
\end{proposition}
\begin{proof}
  By \cite[Theorem 6.31]{RoWe98}, for any $x\in C$ there holds the inclusion
  \begin{equation}
    \label{EqAuxInclTanCone}T_C(x)\subset\{u\mv\nabla\Phi(x)u\in T_A\big(\Phi(x)\big)\}.
  \end{equation}
  Consider $u\in T^\sharp_C(\xb)$ together with sequences $(x_k,u_k)\longsetto{\gph T_C}(\xb,u)$. Then \[(\Phi(x_k),\nabla\Phi(x_k)u_k)\to(\Phi(\xb),\nabla \Phi(\xb)u)\]
  and $\nabla \Phi(x_k)u_k\in T_A\big(\Phi(x_k)\big)$ verifying $\nabla\Phi(\xb)u\in T^\sharp_A\big(\Phi(\xb)\big)$. This proves \eqref{EqCalcTsharp}.
  Now assume that $\nabla\Phi(\xb)$ has full row rank. Then $\Phi$ is metrically regular with some constant $\kappa$ around $\big(\xb,\Phi(\xb)\big)$, see, e.g., \cite[Example 9.44]{RoWe98}. In addition, we can find a neighborhood $U$ of $\xb$ such that $\nabla\Phi(x)$ has full row rank for every $x\in U$ and we conclude from \cite[Exercise 6.7]{RoWe98} that inclusion \eqref{EqAuxInclTanCone} holds with equality for every $x\in U$. Consider $v\in T_A^\sharp\big(\Phi(\xb)\big)$ together with sequences $(y_k,v_k)\longsetto{\gph T_A}(\Phi(\xb),v)$. By metric regularity of $\Phi$, for every $k$ sufficiently large we can find $x_k\in\Phi^{-1}(y_k)$ with $\norm{x_k-\xb}\leq\kappa\norm{y_k-\Phi(\xb)}$ so that $x_k\to\xb$ and $x_k\in U$. Consider $u\in\R^n$ with $\nabla \Phi(\xb)u=v$. For the pseudo-inverse $\nabla\Phi(\xb)^\dag:=\nabla\Phi(\xb)^T\big(\nabla \Phi(\xb)\nabla\Phi(\xb)^T\big)^{-1}$ there holds
   \[u=\nabla\Phi(\xb)^\dag v+\big(I-\nabla\Phi^(\xb)^\dag\nabla\Phi(\xb)\big)u.\]
    Since the pseudo-inverses $\nabla\Phi(x_k)^\dag$ converge to $\nabla \Phi(\xb)^\dag$, we conclude that the sequence
    \[u_k:=\nabla\Phi(x_k)^\dag v_k+\big(I-\nabla \Phi(x_k)^\dag\nabla\Phi(x_k)\big)u\]
     converges to $u$.  Further, since $\nabla \Phi(x_k)u_k=v_k\in T_A\big(\Phi(x_k)\big)$, we have $u_k\in T_C(x_k)$ and $u\in T^\sharp_C(\xb)$ follows. This justifies the inclusion $T^\sharp_C(\xb)\supset\{u\mv \nabla\Phi(\xb)u\in T^\sharp_A\big(\Phi(\xb)\big)\}$ and the proof of the proposition is complete.
\end{proof}

The next calculus rule is essential for the main stability result presented in the fifth section.

Let us consider the situation when $F:\R^n\tto\R^l\times\R^k$ is given via
\begin{equation}\label{EqF}
F(x):= \myvec{G(x)\\H(x)}
\end{equation}
where $G:\R^n\to\R^l$ is a $C^{1}$ function and $H: \R^n\tto\R^k$ has a closed graph.

\begin{proposition}\label{PropCalc1}
Consider $(\xb,\zb)\in\gph H$. Then for the mapping $F$ given by \eqref{EqF} one has:
\begin{enumerate}
\item[(i)]
\begin{align}\label{EqTsharpF}
  T^\sharp_{\gph F}\big(\xb,(G(\xb),\zb)\big)=\big\{\big(u,(\nabla G(\xb)u,w)\big)\mv (u,w)\in T^\sharp_{\gph H}(\xb,\zb)\}
\end{align}
\item[(ii)] If $H$ is \ssstar at $(\xb,\zb)$, then $F$ is \ssstar at $\big(\xb,(G(\xb),\zb)\big)$.
\item[(iii)] Assume that $H$ has the SCD property at $(\xb,\zb)$. Then $F$ has the SCD property at $\big(\xb, (G(\xb),\zb)\big)$ and one has that
\begin{gather}\label{eq-48}
\Sp F\big(\xb, (G(\xb),\zb)\big)=\Big\{\big\{\big(u,(\nabla G(\xb)u,w)\big)\mv (u,w)\in M\big\}\bmv M \in \Sp H(\xb,\zb)\Big\},\\
\label{eq-49}\Sp^* F\big(\xb,(G(\xb),\zb)\big)=\Big\{\{\big((q^*,w^*),\nabla G(\xb)^Tq^*+u^*\big)\mv q^*\in\R^l,\ (w^*,u^*)\in M^*\}\bmv M^*\in\Sp^*H(\xb,\zb)\Big\}.
\end{gather}
\end{enumerate}
\end{proposition}
\begin{proof}
  Let $\tilde H:\R^n\tto\R^l\times\R^k$ be given by $\tilde H(x)=\{0\}\times H(x)$. Then
  \[\gph F=\{(x,p,z)\mv \Phi(x,p,z)\in\gph\tilde H\},\]
  where $\Phi:\R^n\times\R^l\times\R^k\to \R^n\times\R^l\times\R^k$ is given by $\Phi(x,p,z) = (x,p-G(x),z)^T$.
  Note that for every triple $(x,p,z)$ the Jacobian
  \[\nabla \Phi(x,p,z)=\left(\begin{matrix}
    I_n&0&0\\-\nabla G(x)&I_l&0\\0&0&I_k
  \end{matrix}\right)\]
   is nonsingular.

     Ad (i): Obviously there holds  $T^\sharp_{\gph \tilde H}\big(\xb,(0,\zb)\big)=\{\big(u,(0,w)\big)\mv (u,w)\in T^\sharp_{\gph H}(\xb,\zb)\}$. Thus we obtain from Proposition \ref{PropCalcTsharp}
     that
     \begin{align*}T^\sharp_{\gph F}\big(\xb,(G(\xb),\zb)\big)&=\Big\{\big(u,(q,w)\big)\bmv \nabla\Phi(\xb,G(\xb),\zb)(u,q,w)\in T^\sharp_{\gph \tilde H}(\xb,0,\zb)\Big\}\\
     &=\Big\{\big(u,(q,w)\big)\mv (u,w)\in T^\sharp_{\gph H}(\xb,\zb),\ q-\nabla G(\xb)u=0\},\end{align*}
     yielding \eqref{EqTsharpF}.

   Ad (ii): Since $H$ is \ssstar at $(\xb,\zb)$, $\tilde H$ is \ssstar at $\big(\xb,(0,\zb)\big)$. Surjectivity of $\nabla\Phi(\xb, G(\xb),\zb)$ ensures that the mapping $\Phi(\cdot)-\gph\tilde H$ is metrically regular around $\big(\xb,(G(\xb),\zb)\big)$, cf. \cite[Example 9.44]{RoWe98} and therefore metrically subregular as well. Now the claimed statement follows from Proposition \ref{PropSSstarNew}.

   Ad (iii): It is easy to see that $\tilde H$ has the SCD property at $\big(\xb,(0,\zb)\big)$ with
  \begin{gather*}\Sp \tilde H\big(\xb,(0,\zb)\big)=\Big\{\{\big(u,(0,w)\big)\mv (u,w)\in M\}\bmv M\in \Sp H(\xb,\zb)\Big\},\\
  \Sp^* \tilde H\big(\xb,(0,\zb)\big)=\Big\{\{\big((q^*,w^*),u^*\big)\mv q^*\in\R^l,\ (w^*,u^*)\in M^*\}\bmv M^*\in\Sp^*H(\xb,\zb)\Big\}.
  \end{gather*}
  Next we can apply Theorem \ref{ThCalc1} to obtain
   \begin{gather*}
     \Sp F\big(\xb,(G(\xb),\zb)\big)=\{\nabla\Phi(\xb,G(\xb),\zb)^{-1}\tilde M\mv \tilde M\in \Sp \tilde H\big(\xb,(0,\zb)\big)\},\\
     \Sp^* F\big(\xb,(G(\xb),\zb)\big)=\{S_{n(l+k)}\nabla \Phi(\xb,G(\xb),\zb)^TS_{n(l+k)}^T\tilde M^*\mv \tilde M^*\in\Sp^*\tilde H\big(\xb,(0,\zb)\big)\}.
   \end{gather*}
    Straightforward calculations yield that
   \[\nabla \Phi(\xb,G(\xb),\zb)^{-1}=\left(\begin{matrix}I_n&0&0\\\nabla G(\xb)&I_l&0\\0&0&I_k\end{matrix}\right),\
    S_{n(l+k)}\nabla \Phi(\xb,G(\xb),\zb)^TS_{n(l+k)}^T=\left(\begin{matrix}
      I_l&0&0\\0&I_k&0\\\nabla G(\xb)^T&0&I_n
    \end{matrix}\right)\]
    and formulas \eqref{eq-48}, \eqref{eq-49} follow.
\end{proof}

\section{Isolated calmness on a neighborhood of implicit multifunctions}

Given a multifunction $H: \mathbb{R}^{l} \times \mathbb{R}^{k}\rightrightarrows\mathbb{R}^{k}$ with closed graph and a point $\big((\bar{x},\bar{y}),\bar{z}\big)\in
 \gph H$, then the relation
 \begin{equation}\label{eq-51}
\gph \Sigma = H^{-1}(\bar{z})
 \end{equation}
defines the so-called {\em implicit multifunction} $\Sigma:\mathbb{R}^{l}\rightrightarrows\mathbb{R}^{k}$. Our aim is now to ensure a certain stability property of $\Sigma$ around $(\bar{x},\bar{y})$ by imposing suitable assumptions on $H$ around $(\bar{x},\bar{y},\bar{z})$. Usually one puts $\bar{z}=0$ so that
\begin{equation}\label{eq-52}
\gph \Sigma = \{(x,y)\mv 0 \in H(x,y)\}.
\end{equation}
It is easy to see that any stability property of $\Sigma$ around $(\bar{x},\bar{y})$ is inherited by the same stability property of the inverse to the ``extended'' mapping $F:
\mathbb{R}^{l+k}\rightrightarrows\mathbb{R}^{l+k}$ given by
\begin{equation}\label{eq-53}
F(x,y)= \myvec{x\\ H(x,y)}
\end{equation}
%%%
around $\big((\bar{x},\bar{y}),(\bar{x},0)\big)$. In fact, in this way, e.g., the classical Implicit Function Theorem or the Clarke Implicit Function Theorem have been proved. Alternatively, one can combine a suitable characterization of the examined property in terms of a generalized derivative with the available calculus, as shown, e.g., in \cite[Section 4.3]{M06} or \cite[Section 4]{GfrOut16} in case of the Aubin property. In our approach we will use the mapping (\ref{eq-53}) along with Theorem \ref{ThCharRegByDer}(iii) and Corollary \ref{CorStrMetrSubreg}.

\begin{theorem}\label{Th5.1}
Consider the inclusion $0 \in H(x,y)$ and a point $\big((\bar{x},\bar{y}) ,0\big)\in \gph H$. Then any of the following two conditions ensures the isolated calmness property of the respective implicit solution map $\Sigma$ around $(\xb,\yb)$.
\begin{enumerate}
\item[(i)]
\begin{align}\label{EqIsolHSharp}
  0\in D^\sharp H\big((\xb,\yb),0\big)(0,v)\ \Rightarrow\ v=0.
\end{align}
\item[(ii)] The mapping $H$ has both the SCD property and the \ssstar property around $\big((\xb,\yb),0\big)$ and either the implication
\begin{equation}\label{eq-54}
\big((0,v),0\big)\in L\ \Rightarrow\ v=0\end{equation}
holds for all $L \in  \mathcal{S}H \big((\bar{x},\bar{y}),0\big)$, or, equivalently, the implication
\begin{equation}\label{eq-55}
\big(w^*,(u^*,0)\big)\in L^*\ \Rightarrow\ w^*=0,\ u^*=0
\end{equation}
holds for all $L^{*}\in \mathcal{S}^{*}H \big((\bar{x},\bar{y}),0\big)$.
\end{enumerate}
\end{theorem}
\begin{proof}
In the first case we conclude from Proposition \ref{PropCalc1}(i) that the mapping $F$ given by (\ref{eq-53}) fulfills
\[D^\sharp F\big((\xb,\yb),(\xb,0)\big)(u,v)=\{(u,w)\mv w\in D^\sharp H\big((\xb,\yb),0\big)(u,v)\}.\]
Thus it follows from Theorem \ref{ThCharRegByDer}(iii) that condition \eqref{EqIsolHSharp} is equivalent with strong metric subregularity of $F$ around $\big((\xb,\yb),(\xb,0)\big)$ and the claimed isolated calmness of $\Sigma$ around $(\xb,\yb)$ follows.

In the second case, note that by Proposition \ref{PropCalc1}(ii), (iii) the mapping $F$ has the SCD property around $(\xb,\yb)$ and is \ssstar around $(\xb,\yb)$. Further we have
\begin{gather*}
  \Sp F\big((\xb,\yb),(0,0)\big)=\Big\{\{\big((u,v),(u,w)\big)\mv \big((u,v),w\big)\in L\}\bmv L\in\Sp H\big((\xb,\yb),0\big)\Big\},\\
  \Sp^*F\big((\xb,\yb),(0,0)\big)=\Big\{\{\big((q^*,w^*),(q^*+u^*,v^*)\big)\mv q^*\in\R^l,\ \big(w^*,(u^*,v^*)\big)\in L^*\}\bmv L^*\in\Sp^* H\big((\xb,\yb),0\big)\Big\}.
\end{gather*}
Implications \eqref{eq-41} and \eqref{eq-40} now yield conditions \eqref{eq-54}, \eqref{eq-55} which, by  Corollary \ref{CorStrMetrSubreg}, are  equivalent with the strong metric subregularity of $F$ around $\big((\xb,\yb),(\xb,0)\big)$. The proof is complete.
\end{proof}

Theorem \ref{Th5.1} can well be  applied to parameterized GEs. To this aim consider the case when  $H:\mathbb{R}^{l} \times \mathbb{R}^{k}\rightrightarrows\mathbb{R}^{k}$ is given via
\begin{equation}\label{eq-59}
H(x,y):=f(x,y)+Q(x,y),
\end{equation}
where $x\in \mathbb{R}^{l}$ is the {\em perturbation parameter}, $y\in \mathbb{R}^{k} $ is the {\em decision variable}, $f:\mathbb{R}^{l} \times \mathbb{R}^{k}\rightarrow\mathbb{R}^{k}$ is continuously differentiable and $Q:
\mathbb{R}^{l} \times \mathbb{R}^{k}\rightrightarrows\mathbb{R}^{k}$ has a closed graph.

\begin{proposition}\label{PropGE}
Consider the reference point $(\bar{x},\bar{y})\in H^{-1}(0)$ and assume that one of the following conditions hold true:
\begin{enumerate}
\item[(i)]
\begin{align}\label{EqIsolQSharp}
  0\in \nabla_y f(x,y)v+ D^\sharp Q\big((\xb,\yb),-f(\xb,\yb)\big)(0,v)\ \Rightarrow\ v=0.
\end{align}
\item[(ii)] $Q$ has the SCD property around $\big((\xb,\yb), -f(\xb,\yb)\big)$ and is \ssstar on a neighborhood of $\big((\bar{x},\bar{y}),-f(\bar{x},\bar{y})\big)$ and either one of the implications
\begin{equation}\label{eq-71}
\big((0,v),-\nabla_y f(\xb,\yb)v\big)\in M\ \Rightarrow\ v=0 ~\mbox{ for all }~ M \in \Sp Q\big((\bar{x},\bar{y}),-f(\bar{x},\bar{y})\big)
\end{equation}
and
\begin{equation}\label{eq-72}
\big(w^*,(u^*,-\nabla_yf(\xb,\yb)^Tw^*)\big)\in M^*\ \Rightarrow\ w^* = 0 ,\ u^* = 0~\mbox{ for all }~ M^{*} \in \mathcal{S}^{*}Q\big((\bar{x},\bar{y}),-f(\bar{x},\bar{y})\big)
\end{equation}
holds true.
\end{enumerate}
 Then the respective solution mapping $\Sigma:\mathbb{R}^{l}\rightrightarrows\mathbb{R}^{k}$ is isolatedly calm around $(\bar{x},\bar{y})$.
\end{proposition}
\begin{proof}
Clearly
\[
\gph H = \{\big((x,y),z\big)\in \mathbb{R}^{l} \times \mathbb{R}^{k} \times \mathbb{R}^{k} \mv \Phi\big((x,y),z\big)\in\gph Q\}\mbox{ with } \Phi\big((x,y),z\big)=\big((x,y),z-f(x,y)\big)^T
\]
so that we can apply Proposition \ref{PropCalcTsharp} and Theorem \ref{ThCalc1} to obtain
\begin{gather*}
T^\sharp_{\gph H}\big((\xb,\yb),0\big)=\nabla\Phi\big((\xb,\yb),0\big)^{-1} T^\sharp_{\gph Q}\big(\Phi\big((\xb,\yb),0\big)\big),\\
\Sp H\big((\xb,\yb),0\big)=\nabla\Phi\big((\xb,\yb),0\big)^{-1} \Sp Q\big(\Phi\big((\xb,\yb),0\big)\big),\\
\Sp^* H\big((\xb,\yb),0\big)= \{L^{*} \in \Z_{k(l+k)} \mv L^{*} = S_{(l+k)k}\nabla\Phi\big((\xb,\yb),0\big)^T S^T_{(l+k)k}M^* \mbox{ with } M^* \in \Sp^{*} Q \big(\Phi\big((\xb,\yb),0\big)\big)\}.
\end{gather*}
Straightforward calculations yield
\[\nabla\Phi\big((\xb,\yb),0\big)^{-1}=\left(\begin{matrix}I_{l+k}&0\\\nabla f(\xb,\yb)&I_k\end{matrix}\right), \quad
S_{(l+k)k}\nabla\Phi\big((\xb,\yb),0\big)^T S^T_{(l+k)k}=\left(\begin{matrix}I_k&0\\\nabla f(\xb,\yb)^T&I_{l+k}\end{matrix}\right)\]
and we arrive at the formulas
\begin{align*}
&D^\sharp H\big((\xb,\yb),0\big)(u,v)=\{ \nabla_xf(\xb,\yb)u+\nabla_yf(\xb,\yb)v+w\mv w\in D^\sharp Q\big((\xb,\yb),-f(\xb,\yb)\big)(u,v)\},\\
&\Sp H\big((\xb,\yb),0\big)=\Big\{\{\big((u,v), \nabla_xf(\xb,\yb)u+\nabla_yf(\xb,\yb)v+w\big)\mv \big((u,v),w\big)\in M\}\bmv M\in \Sp Q\big((\xb,\yb),-f(\xb,\yb)\big)\Big\}\end{align*}
and
\begin{align*}
&\lefteqn{\Sp^* H\big((\xb,\yb),0\big)}\\
&=\Big\{\big\{\big(w^*,(\nabla_xf(\xb,\yb)^Tw^*+u^*,\nabla_yf(\xb,\yb)^Tw^*+v^*)\big)\mv (w^*,(u^*,v^*))\in M^*\big\}\bmv M^*\in \Sp^* Q\big((\xb,\yb),-f(\xb,\yb)\big)\Big\}.
\end{align*}
Conditions \eqref{EqIsolHSharp}, \eqref{eq-54}, \eqref{eq-55} thus read as
\begin{gather*}
\left.\begin{array}{l}u=0,\ \nabla_xf(\xb,\yb)u+\nabla_yf(\xb,\yb)v+w=0\\
w\in D^\sharp Q\big((\xb,\yb),-f(\xb,\yb)\big)(u,v) \end{array}\right\}\Rightarrow v=0,\\
\left.\begin{array}{l}u=0,\ \nabla_xf(\xb,\yb)u+\nabla_yf(\xb,\yb)v+w=0\\
\big((u,v),w\big)\in M \end{array}\right\}\Rightarrow v=0,\ M\in\Sp Q\big((\xb,\yb),-f(\xb,\yb)\big),\\
\left.\begin{array}{l}\nabla_yf(\xb,\yb)^Tw^*+v^*=0\\ \big(w^*,(u^*,v^*)\big)\in M^*\end{array}\right\}\Rightarrow w^*=0, \nabla_xf(\xb,\yb)^Tw^*+u^*=0,\ M^*\in\Sp^* Q\big((\xb,\yb),-f(\xb,\yb)\big),\end{gather*}
which are equivalent to \eqref{EqIsolQSharp}, \eqref{eq-71} and  \eqref{eq-72}. This completes the proof.
\end{proof}
Recall that for SCD mappings $Q$ having the \ssstar property any of the three conditions \eqref{EqIsolQSharp}, \eqref{eq-71} and  \eqref{eq-72} is equivalent to the strong metric subregularity on a neighborhood of the mapping $F$ given by \eqref{eq-53} and thus the conditions \eqref{EqIsolQSharp}, \eqref{eq-71} and  \eqref{eq-72} are equivalent. Whereas
\eqref{eq-72} is a dual formulation of \eqref{eq-71}, conditions \eqref{EqIsolQSharp} and \eqref{eq-71} might look quite different. Let us shed light on this issue by the following application of Proposition \ref{PropGE} to parameterized variational inequalities with polyhedral constraint sets.

Consider the GE
\begin{equation}\label{EqHcompl}0\in H(x,y):=f(x,y)+N_D\big(g(x,y)\big),\end{equation}
where $f,g:\R^l\times\R^k\to\R^k$ are continuously differentiable and $D\subset\R^k$ is a convex polyhedral set, and let $0\in H(\xb,\yb)$. In what follows we denote by
\[\K_D(d,d^*):=T_D(d)\cap[d^*]^\perp,\ (d,d^*)\in\gph N_D\]
the critical cone to $D$ at $d$ for $d^*$.

Our further development makes use of the following  statements.
\begin{proposition}\label{PropN_D}
  Let $D\subset\R^k$ be a convex polyhedral set. Then the normal cone mapping $N_D(\cdot)$ is an SCD mapping, which is \ssstar at every point of its graph. Further, for every point $(d,d^*)\in\gph N_D$ there holds
  \begin{align}\label{EqSpN_D}
    \Sp N_D(d,d^*)&= \Sp^* N_D(d,d^*)=\{(\F-\F)\times(\F-\F)^\perp\mv \mbox{$\F$ is face of $\K_D(d,d^*)$}\},
  \end{align}
   and the outer limiting tangent cone $T^\sharp_{\gph N_D}(d,d^*)$ is the union of all sets $\gph N_{\F_1-\F_2}$, where $\F_1,\F_2$ are closed faces of $\K_D(d,d^*)$ with $\F_2\subset\F_1$.
\end{proposition}
\begin{proof}
  Since the normal cone mapping $N_D$ is the subdifferential mapping of the convex lsc function $\delta_D$, it is an SCD mapping by \cite[Corollary 3.28]{GfrOut22a}. Further, since $\gph N_D$ is the union of finitely many convex polyhedral sets, $N_D$ is \ssstar at every point of its graph by Propsoition \ref{PropSSstar}. Formula \eqref{EqSpN_D} can be found in  \cite[Example 3.29]{GfrOut22a} and there remains to show the representation of $T^\sharp_{\gph N_D}(d,d^*)$.

  For any $(d',d'^*)\in\gph N_D$ there holds $T_{\gph N_D}=\gph N_{\K_D(d',d'^*)}$, cf. \cite[Lemma 2E4]{DoRo14}. Further, by the Critical Superface Lemma \cite[Lemma 4H.2]{DoRo14}, for every sufficiently small neighborhood $W$ of $(d,d^*)$ the collection of all critical cones $\K_D(d',d'^*)$, $(d',d'^*)\in\gph N_D\cap W$ coincides with the collection of the so-called critical superfaces $\F_1-\F_2$, where $\F_1,\F_2$ are faces of the critical cone $\K_D(d,d^*)$ with $\F_2\subset\F_1$. Now consider a quadruple $\big((d,d^*),(e,e^*)\big)$ satisfying the relation $(e,e^*)\in T^\sharp_{\gph N_D}(d,d^*)$ together with sequences $\big((d_k,d_k^*),(e_k,e_k^*)\big)\to\big((d,d^*),(e,e^*)\big)$ with $(e_k,e_k^*)\in T_{\gph N_D}(d_k,d_k^*)$. Since the convex polyhedral set $D$ has only finitely many faces, after possibly passing to a subsequence we can assume that there are two faces $\F_1, \F_2$ of $\K_D(d,d^*)$ with $\F_2\subset\F_1$ such that $\K_D(d_k,d_k^*)=\F_1-\F_2$ for all $k$. Thus, $(e_k,e_k^*)\in T_{\gph N_D}(d_k,d_k^*)=\gph N_{\K_D(d_k,d_k^*)}=\gph N_{\F_1-\F_2}$ for all $k$ and $(e,e^*)\in \gph N_{\F_1-\F_2}$ follows. Conversely, let $\F_2\subset\F_1$ be two faces of $\K_D(d,d^*)$ and let $(e,e^*)\in \gph N_{\F_1-\F_2}$. Then there exists some sequence $(d_k,d_k^*)\longsetto{\gph N_D}(d,d^*)$ with $\K_D(d_k,d_k^*)=\F_1-\F_2$ $\forall k$ so that $(e,e^*)\in \gph N_{\F_1-\F_2}=T_{\gph N_D}(d_k,d_k^*)$ $\forall k$ implying $(e,e^*)\in T^\sharp_{\gph N_D}(d,d^*)$. The statement has been established.
\end{proof}
\begin{proposition}
  In the setting of \eqref{EqHcompl}, assume that $g(\xb,\yb)\in D$ and the Jacobian $\nabla g(\xb,\yb)$ has full row rank $k$. Then the mapping $Q(x,y):\R^l\times\R^k\tto\R^k$ given by $Q(x,y)=N_D\big(g(x,y)\big)$ has the SCD property around $\big((\xb,\yb),d^*\big)$ and is \ssstar around $\big((\xb,\yb),d^*\big)$ for every $d^*\in N_D\big(g(\xb,\yb)\big)$. Further one has
  \begin{align*}&\lefteqn{\Sp Q\big((\xb,\yb),d^*\big)=}\\
  &\Big\{\{\big((u,v),e^*\big)\mv \myvec{\nabla g(\xb,\yb)(u,v)\\ e^*}\in(\F-\F)\times(\F-\F)^\perp\}\bmv \mbox{$\F$ is face of $\K_D(g(\xb,\yb),d^*)$}\Big\}
  \end{align*}
  and
  \begin{align}\label{EqTsharpQcompl}&\lefteqn{T^\sharp_{\gph Q}\big((\xb,\yb),d^*\big)=}\\
  \nonumber&\Big\{\big((u,v),e^*\big)\bmv \myvec{\nabla g(\xb,\yb)(u,v)\\ e^*}\in\gph N_{\F_1-\F_2} \mbox{ for faces $\F_1,\F_2$ of $\K_D(g(\xb,\yb),d^*)$ with $\F_2\subset\F_1$}\Big\}.
  \end{align}
\end{proposition}
\begin{proof}
  Obviously
  \[\gph Q=\{\big((x,y),d^*\big)\mv \Phi(x,y,d^*):=\myvec{g(x,y)\\d^*}\in\gph N_D\}.\]
  The full-rank assumption imposed on $\nabla g(\xb,\yb)$ ensures that $\nabla g(x,y)$ has full row rank for all $(x,y)$ belonging to some neighborhood $U$ of $(\xb,\yb)$ and it follows that $\nabla \Phi(x,y,d^*)$ has full row rank for all $(x,y,d^*)\in U\times\R^k$. Hence, by Theorem \ref{ThCalc1}, for any $(x,y,d^*)\in\gph Q\cap U\times\R^k$ the mapping $Q$ has the SCD property at $\big((x,y),d^*\big)$ and, together with \eqref{EqSpN_D},
  \begin{align*}&\lefteqn{\Sp Q\big((x,y),d^*\big)=}\\
  &\Big\{\{\big((u,v),e^*\big)\mv \myvec{\nabla g(x,y)(u,v)\\ e^*}\in(\F-\F)\times(\F-\F)^\perp\}\bmv \mbox{$\F$ is face of $\K_D(g(x,y),d^*)$}\Big\}.
  \end{align*}
  Further, the mapping $(x,y,d^*)\tto \Phi(x,y,d^*)-\gph N_D$ is metrically regular around the point $\big((x,y,d^*),(0,0)\big)$ by  \cite[Example 9.44]{RoWe98} and consequently also metrically subregular. This allows us to invoke Proposition \ref{PropSSstarNew} in order to guarantee the \ssstar property of $Q$ at $\big((x,y),d^*\big)$. Finally, formula \eqref{EqTsharpQcompl} follows from Proposition \ref{PropCalcTsharp} and Lemma \ref{PropN_D}.
\end{proof}
Taking into account our above considerations about the equivalence of \eqref{EqIsolQSharp}, \eqref{eq-71} and the strong metric subregularity of $F$ on a neighborhood  for \ssstar SCD mappings $Q$, we arrive at the following result.
\begin{proposition}\label{PropSurprise}
  In the setting of \eqref{EqHcompl}, assume that $\big((\xb,\yb),0\big)\in\gph H$ and that the Jacobian $\nabla g(\xb,\yb)$ has full row rank $k$. Then the following statements are equivalent.
  \begin{enumerate}
    \item[(i)]The mapping $F:\R^l\times\R^k\tto\R^l\times\R^k$ given by
    \[F(x,y)=\myvec{x\\f(x,y)+N_D\big(g(x,y)\big)}\]
    is strongly metrically subregular around $\big((\xb,\yb),(\xb,0)\big)$.
    \item[(ii)]The implication
    \begin{equation}\label{EqCompl}
    \left.\begin{array}{l}
    \nabla_yg(\xb,\yb)v\in \F-\F\\
    -\nabla_yf(\xb,\yb)v\in (\F-\F)^\perp
    \end{array}\right\}\ \Rightarrow\ v=0
\end{equation}
holds for every face $\F$ of the critical cone $\K_D\big(g(\xb,\yb), -f(\xb,\yb)\big)$.
    \item[(iii)] The implication
  \begin{align}\label{EqCompl1}
  \left.\begin{array}{l}\nabla_y g(\xb,\yb)v\in\F_1-\F_2\\
    -\nabla_y f(\xb,\yb)v\in(\F_1-\F_2)^\circ\\
    \skalp{\nabla_yg(\xb,\yb)v,-\nabla_yf(\xb,\yb)v}=0
    \end{array}\right\}\ \Rightarrow\ v=0
  \end{align}
  holds for every pair $\F_1,\F_2$ of faces of the critical cone $\K_D\big(g(\xb,\yb),-f(\xb,\yb)\big)$ with $\F_2\subset\F_1$.
  \end{enumerate}
\end{proposition}
This result is quite surprising since the implications in (ii) are only  a proper subset of those in (iii) with $\F_1=\F_2$ and it is by no means evident why the remaining implications in (iii) with $\F_2\not=\F_1$ are superfluous. These considerations demonstrate that for testing the strong metric subregularity on a neighborhood of semismooth* SCD mappings the outer limiting graphical derivative might be much too large.

Concerning the isolated calmness of the solution map $\Sigma$ related to \eqref{EqHcompl}, we arrive at the following result.

\begin{proposition}\label{PropCompl}
In the setting of \eqref{EqHcompl}, let $0\in H(\xb,\yb)$. If the implication
\begin{equation*}
  \left.\begin{array}{l}
    \nabla_yg(\xb,\yb)v\in \F-\F\\
    -\nabla_yf(\xb,\yb)v\in (\F-\F)^\perp
  \end{array}\right\}\ \Rightarrow\ v=0
\end{equation*}
holds for every face $\F$ of the critical cone $\K_D\big(g(\xb,\yb), -f(\xb,\yb)\big)$, then the respective solution mapping $\Sigma:\mathbb{R}^{l}\rightrightarrows\mathbb{R}^{k}$ is isolatedly calm around $(\bar{x},\bar{y})$.
\end{proposition}
\begin{proof}
    If $\nabla g(\xb,\yb)$ has full row rank, the assertion follows from Proposition \ref{PropSurprise}.

  If the Jacobian $\nabla g(\xb,\yb)$ does not possess full row rank, we simply consider the generalized equation
  \[0\in \tilde H\big((x,p),y\big)=\tilde f\big((x,p),y\big)+N_D\big(\tilde g\big((x,p),y\big)\big)\]
  where $\tilde f,\tilde g:\R^l\times\R^k\times\R^k\to\R^k$ are given by
  \[\tilde f\big((x,p),y\big)=f(x,y),\ \tilde g\big((x,p),y\big)=g(x,y)-p.\]
  Since the Jacobian $\nabla \tilde g\big((x,p),y\big)$ has full row rank and $\nabla_y \tilde g\big((x,p),y\big)=\nabla_yg(x,y)$, $\nabla_y \tilde f\big((x,p),y\big)=\nabla_yf(x,y)$ for all $(x,p,y)$, we can conclude that the respective solution mapping $\tilde \Sigma:\R^l\times\R^k\to \R^l$ is isolatedly calm around $\big((\xb,0),\yb\big)$ and, together with the observation that $\Sigma(x)=\tilde\Sigma(x,0)$ $\forall x\in\R^l$, the isolated calmness of $\Sigma$ around $(\xb,\yb)$ follows.
\end{proof}

Let us illustrate the above conditions via a simple academic example.
\begin{example}
Let $l=k=1$ and consider the parameterized GE \eqref{EqHcompl}, where $f(x,y)=-y$, $g(x,y)=y-x$ and $D=\R_+$. With $(\xb,\yb)=(0,0)$ we observe that all the assumptions of Proposition \ref{PropCompl} are fulfilled,
\[\K_D\big(g(\xb,\yb),-f(\xb,\yb)\big)=T_D\big(g(\xb,\yb)\big)=\R_+,\]
and one has to consider the faces $\F_1=\R_+$, $\F_2=\{0\}$. We have thus to check the validity of the implications
\begin{gather*}\myvec{v\\v}\in(\F_1-\F_1)\times (\F_1-\F_1)^\perp=\R\times\{0\}\ \Rightarrow\ v=0\\
\myvec{v\\v}\in(\F_2-\F_2)\times (\F_2-\F_2)^\perp=\{0\}\times\R\ \Rightarrow\ v=0,
\end{gather*}
which are evidently fulfilled. Consequently the respective solution mapping $\Sigma$
is isolatedly calm around $(0,0)$. This conclusion is correct because, as one can easily compute,
\[
\Sigma(x)=\begin{cases}\{0\}\cup \{x\} &\mbox{ if $x \leq 0$,}\\
\emptyset&\mbox{otherwise.}
\end{cases}\]
Note that $\Sigma$ does not have the Aubin property around $(0,0)$.\hfill{$\triangle$}
\end{example}

Finally, let us compare condition (\ref{eq-55}) with a standard criterion for the Aubin property of $\Sigma$ around the reference point. On the basis of the theory from \cite[Chapter 4]{M06} one obtains the following result.
\begin{proposition}
Consider the inclusion $0 \in H(x,y)$ and assume that the implication
\begin{equation}\label{eq-100}
(u^*,0)\in D^{*}H\big((\xb,\yb),0\big)(w^*)\Rightarrow w^*=0,\ u^*=0
\end{equation}
holds true.
 Then $\Sigma$ has the Aubin property around $(\bar{x},\bar{y})$.
\end{proposition}

Note that condition (\ref{eq-100}) ensures both a qualification condition needed to compute the coderivative of $\Sigma$ and the satisfaction of the Mordukhovich criterion $D^{*}\Sigma(\bar{x},\bar{y})(0)=\{0\}$. Since $L^*\subset \gph D^{*}H(\bar{x},\bar{y},0)$ $\forall L^*\in \Sp H\big((\xb,\yb),0\big)$, %$\mathcal{S}^{*}H(\bar{x},\bar{y}, 0)\subset \gph D^{*}H(\bar{x},\bar{y},0)$,
 it follows that for $H$ being SCD and \ssstar around $(\bar{x},\bar{y})$ condition (\ref{eq-100}) implies not only the Aubin property but also the isolated calmness of $\Sigma$ around $(\bar{x},\bar{y})$. This is an important fact emphasizing the importance of the SCD and \ssstar property in stability issues. Observe that the conjunction of the Aubin and the isolated calmness property represents a new useful stability notion, where the isolated calmness specifies the nature of Lipschitzian behavior and the Aubin property ensures the non-emptiness of a localization.

\section{Conclusion}
As explained in Corolary 4.3, graphically Lipschitzian mappings of dimension $n$ are SCD mappings for which both SC limiting derivatives can be computed. For GEs with such multi-valued parts thus the respective conditions \eqref{eq-71} and  \eqref{eq-72} can be used in a large number of parameterized GEs corresponding, e.g., to variational inequalities of the 2nd kind, hemivariational inequalities or implicit complementarity problems.

In \cite{GfrOut22a} one finds also a relationship between SCD and strong metric regularity. This indicates that in some cases the SC limiting derivatives could be used also to ensure that an implicitly defined mapping has a single-valued and Lipschitzian localization around the reference point. This task we postpone to a future research.
\section*{Acknowledgements}
This paper is dedicated to Roger J.-B. Wets on the ocassion of his 85th birthday. Apart from many excellent results Roger coauthored the monograph \cite{RoWe98} which became an indispensable part of our equipment in any work connected with modern variational analysis.

The second author expresses his gratitude for the support from the  Grant Agency of the Czech Republic, Project 21- 06569K,  and from the Australian Research Council,  Project DP160100854.


\begin{thebibliography}{99}
\bibitem{AubFra90}
{\sc J.~P. Aubin, H. Frankowska}, {\em Set-Valued Analysis}, Birkh\"auser, Boston, 1990.
%
\if{
\bibitem{BoDaLe09}
{\sc J. Bolte, A. Daniilidis, A. Lewis}, {\em Tame functions are semismooth}, Math. Program. Ser. B, 117 (2009), pp.~5--19.
}\fi
%
\bibitem{DoRo96}{\sc A.~L. Dontchev, R.~T. Rockafellar}, {\em  Characterizations of strong
    regularity for variational inequalities over polyhedral convex sets},
SIAM J. Optim., 6 (1996), pp.~1087--1105.
%
\bibitem{DoRo14}{\sc A.~L. Dontchev, R.~T. Rockafellar}, {\em  Implicit Functions and Solution
    Mappings}, Springer, Heidelberg, 2014.
%
\bibitem{GfrOut16}
{\sc H. Gfrerer, J.~V. Outrata}, {\em On Lipschitzian properties of implicit multifunctions}, SIAM J. Optim. 26 (2016), pp.~2160--2189.
%
\bibitem{GfrOut21}
{\sc H. Gfrerer, J.~V. Outrata}, {\em On a semismooth* Newton method for solving generalized equations}, SIAM J. Optim. 31 (2021), pp.~489--517.
%
\bibitem{GfrOut22a}
{\sc H. Gfrerer, J.~V. Outrata}, {\em On (local) analysis of multifunctions via subspaces contained
in graphs of generalized derivati}, J. Math. Annal. Appl. 508 (2022), \url{https://doi.org/10.1016/j.jmaa.2021.125895}
%
%
\bibitem{GfrYe17}{\sc H. Gfrerer, J.J. Ye}, {\em New constraint qualifications for
    mathematical programs with equilibrium constraints via variational analysis},  SIAM J. Optim., 27 (2017), pp. 842-865.
%
\bibitem{Gow04}
{\sc M.~S. Gowda}, {\em  Inverse and implicit function theorems for H-differentiable and semismooth functions},
Optim. Methods Softw. 19  (2004), pp.~443--461.
%
%
\bibitem{Jou07}
{\sc A. Jourani}, {\em Radiality and semismoothness}, Control and Cybernetics 36 (2007), pp.~669--680.
%
%
\bibitem{M06}
{\sc B.~S. Morkukhovich}, {\em Variational Analysis and Generalized Differentiation I: Basic Theory}, Springer, Berlin, 2006.
%
%
\bibitem{Mo18}{\sc B.~S. Mordukhovich}, {\em Variational Analysis and Applications}, Springer, Cham, 2018.
%
%
\bibitem{QiSun93}
{\sc L. Qi, J. Sun}, {\em A nonsmooth version of Newton's method}, Math. Program., 58 (1993),
pp.~353--367.
%
\bibitem{Ro81}{\sc S.~M. Robinson}, {\em Some continuity properties of polyhedral multifunctions}, Math. Prog. Study, 14 (1981), pp.~206--214.
%%
\bibitem{Ro70}{\sc R.~T. Rockafellar}, {\em Convex analysis}, Princeton, New Jersey, 1970.
%
\bibitem{RoWe98}
{\sc R.~T. Rockafellar, R.~J.-B. Wets }, {\em Variational Analysis}, Springer, Berlin, 1998.

\end{thebibliography}
\end{document}